\documentclass[11pt]{article}
\usepackage{a4wide}

\usepackage[a4paper, margin=2.3cm]{geometry}
\usepackage{amsmath,amssymb,amsthm, comment, enumitem, array}
\usepackage{color}
\usepackage{parskip}
\usepackage{hyperref}
\usepackage{parskip}
\usepackage{setspace}
\usepackage{graphicx}
\usepackage{mathtools}
\usepackage[thinc]{esdiff}
\usepackage[euler]{textgreek}

\usepackage{subfig}

\newtheorem{theorem}{Theorem}[section]
\newtheorem{proposition}[theorem]{Proposition}
\newtheorem{remark}{Remark}[section]
\newtheorem{corollary}[theorem]{Corollary}
\newtheorem{lemma}[theorem]{Lemma}
\newtheorem{question}[theorem]{Question}

\newtheorem{definition}[theorem]{Definition}

\newtheorem*{definition*}{Definition}

\def\mike#1{\noindent
\textcolor{green}
{\textsc{(Mike:}
\textsf{#1})}}

\hypersetup{
    colorlinks=true,
    linkcolor = blue,
    citecolor=magenta,
    urlcolor=cyan,
}

\usepackage[capitalise, nameinlink]{cleveref}
\crefname{equation}{}{}
\crefname{figure}{{\sc Figure}}{{\sc Figure}}
\crefname{subsection}{Subsection}{Subsections}

\begin{document}

\title{VC-dimension and pseudo-random graphs}
\author{Thang Pham\thanks{University of Science, Vietnam National University, Hanoi. Email: phamanhthang.vnu@gmail.com}\and Steven Senger\thanks{Department of Mathematics, Missouri State University. Email: StevenSenger@MissouriState.edu}\and Michael Tait\thanks{Department of Mathematics and Statistics, Villanova University. Email: michael.tait@villanova.edu}\and Nguyen Thu-Huyen\thanks{Fulbright University Vietnam. Email: huyen.nguyen.190033@student.fulbright.edu.vn}}
\maketitle
\begin{abstract}
Let $G$ be a graph and $U\subset V(G)$ be a set of vertices. For each $v\in U$, let $h_v\colon U\to \{0, 1\}$ be the function defined by
\[h_v(u)=\begin{cases} &1 ~\mbox{if}~u\sim v, u\in U\\&0 ~\mbox{if}~u\not\sim v, u\in U\end{cases},\]
and set $\mathcal{H}(U):=\{h_v\colon v\in U\}$. The first purpose of this paper is to study the following question: What families of graphs $G$ and what conditions on $U$ do we need so that the VC-dimension of $\mathcal{H}(U)$ can be determined? We show that if $G$ is a pseudo-random graph, then under some mild conditions, the VC dimension of $\mathcal{H}(U)$ can be bounded from below. Specific cases of this theorem recover and improve previous results on VC-dimension of functions defined by the well-studied distance and dot-product graphs over a finite field.
\end{abstract}
\section{Introduction}
In this paper, we study the VC-dimension of sets of functions that are defined by graph adjacency. We start with the requisite definitions.
\begin{definition}
Let $\mathcal{H}$ be a collection of functions from $V$ to $\{0, 1\}$. We say that $\mathcal{H}$ shatters a finite set $X\subset V$ if the restriction of $\mathcal{H}$ to $X$ yields every possible function from $X$ to $\{0, 1\}$.
\end{definition}

\begin{definition}
Let $\mathcal{H}$ be a collection of functions from $V$ to $\{0, 1\}$. The Vapnik-Chervonenkis dimension (in short, VC-dimension) of $\mathcal{H}$ is $d$ if any only if there exists a set $X\subset V$ of size $d$ that is shattered by $\mathcal{H}$, and no subset of $X$ of size $d+1$ is shattered by $\mathcal{H}$.
\end{definition}

It is mentioned in the book \cite{matousek2013lectures} that ``the VC-dimension can be determined without great difficulty in several simple cases, such as for half-spaces or balls in $\mathbb{R}^d$, but for only slightly more complicated families its computation becomes challenging." 
Before stating our main results, we need to set up notation as follows.

Let $G=(V, E)$ be a graph with the vertex set $V$ and the edge set $E$. For two vertices $u, v\in V$, by $u\sim v$, we mean there is an edge between $u$ and $v$. 

Let $U\subset V$ be a set of vertices. For each $v\in U$, let $h_v\colon U\to \{0, 1\}$ be the function defined by
\[h_v(u)=\begin{cases} &1 ~\mbox{if}~u\sim v, u\in U\\&0 ~\mbox{if}~u\not\sim v, u\in U\end{cases},\]
and set
\[\mathcal{H}(U):=\{h_v\colon v\in U\}.\]
The following question appears to us to be natural and fundamental. 
\begin{question}\label{qs}
For which families of graphs and under what conditions on $U$ can we determine the VC-dimension of $\mathcal{H}(U)$?
\end{question}

One of the primary motivations for this question comes from a series of recent papers \cite{VC1, VC2, IMS, MSW} in which the authors studied graphs $G$ with $V(G)=\mathbb{F}_q^t$, where $\mathbb{F}_q$ is the finite field of order $q$, and there is an edge between two vertices $x$ and $y$ if $P(x, y)=1$ for some polynomials $P$. More precisely, if $t=2$ and $P(x, y)=(x_1-y_1)^2+(x_2-y_2)^2$, then Fitzpatrick, Iosevich, McDonald, and Wyman \cite{VC2} showed that for any subset $U\subset \mathbb{F}_q^2$ of size at least $Cq^{15/8}$, the set $\mathcal{H}(U)$ has VC-dimension equal to $3$. They also showed that for $|U| \geq Cq^{3/2}$ the set $\mathcal{H}(U)$ has VC-dimension at least $2$. They asked specifically about the gap between these two exponents and if one can determine the VC-dimension for sets of size $o(q^{15/8})$. When $t=3$ and $P(x, y)=x_1y_1+x_2y_2+x_3y_3$, Iosevich, McDonald, and Sun \cite{IMS} showed that in this graph any subset $U\subset \mathbb{F}_q^3$ of size at least $Cq^{11/4}$ has the property that $\mathcal{H}(U)$ has VC-dimension equal to $3$, and that any set with $|U| \geq C q^{5/2}$ has VC-dimension at least $2$. They similarly remarked that they ``do not know to what extent the exponent $\frac{11}{4}$... and ... $\frac{5}{2}$ are sharp, but we know that neither exponent can fall below $2$." 
Our main goal in this paper is to develop a framework for these questions. Ideally, we would like to be able to prove theorems like the concrete examples given above as specific applications of a general theorem. We make partial progress towards this goal. More precisely, we will address Question \ref{qs} for $(n, d, \lambda)$-graphs, i.e. graphs with $n$ vertices, regular of degree $d$, and the second largest eigenvalue is at most $\lambda$ in absolute value. Our main results can be stated as follows. 

\begin{theorem}\label{thm1}
Let $G$ be an $(n, d, \lambda)$-graph with $d=o(n)$. Let $U$ be a vertex set. Assume that $\lambda n/d = o(|U|)$, then the VC-dimension of $\mathcal{H}(U)$ is at least $2$. 
\end{theorem}

If more conditions are allowed, then one can be guaranteed a larger dimension in the next theorem. Although in the next theorem the dimension is increasing by only one from Theorem \ref{thm1}, the proof becomes much more difficult. 

\begin{theorem}\label{thm2}
Let $G$ be an $(n, d, \lambda)$-graph with $d=o(n)$. Let $U$ be a vertex set. Assume that the following properties hold.
\begin{enumerate}
    \item\label{condition} Given any three vertices $v_1, v_2, v_3 \in U$, we can find three vertices $u_1, u_2, u_3 \in U$ such that $u_i\sim v_j$ if and only if $i=j$,
    \item The size of $U$ satisfies \[|U| \geq C\max\left\lbrace \lambda^{2/3}(n/d)^2, \lambda (n/d)^{13/7}\right\rbrace.\]
\end{enumerate}
where $C$ is an absolute constant. Then the VC-dimension of $\mathcal{H}(U)$ is at least $3$. 
\end{theorem}




We make some remarks before proceeding. First, since $|\mathcal{H}(U)| = |U|$, the trivial upper bound gives that the VC-dimension is at most $\log_2(|U|)$, and sets with VC-dimension $2$ or $3$ are far from this. Second, when showing that the VC-dimension of a set is at least $3$, we require condition \ref{condition} in the hypotheses of Theorem \ref{thm2}. It would be very interesting to get rid of this condition or replace it with a more natural one that applies to $(n, d, \lambda)$-graphs in a more general way. Finally, in all of the previously studied graphs defined by distances or dot-products, there are geometric reasons that make finding a corresponding upper bound on the VC-dimension easy. We do not know how to give upper bounds on VC-dimension in a general way for $(n,d, \lambda)$-graphs. 



The main idea to prove Theorems \ref{thm1} and \ref{thm2} is to find a subset of $2$ or $3$ vertices respectively that are shattered by a large enough subset $U$. From the way the functions are defined, we must find sets of vertices with prescribed adjacencies and non-adjacencies; the backbone of our proof will be to show that in a suitably large subset of a pseudo-random graph, we can find and count various subgraphs. We find these results interesting in their own right, as finding or counting certain subgraphs in pseudo-random graphs has a long history of study. For example, a representative result is the following classical theorem of Alon \cite[Theorem 4.10]{KS}.

\begin{theorem}\label{alon counting}
Let $H$ be a fixed graph with $r$ edges, $s$ vertices and maximum degree $\Delta$, and let $G$ be an $(n,d,\lambda)$-graph with $d\leq 0.9n$. Then for any subset $U$ with $\lambda \left(\frac{n}{d}\right)^\Delta = o(|U|)$, the number of copies of $H$ in $U$ is 
\[
(1+o(1))\frac{|U|^s}{|Aut(H)|}\left(\frac{d}{n}\right)^r
\]
\end{theorem}
We will state our subgraph counting theorems in the coming sections and give some discussion of them, including comparison with Theorem \ref{alon counting}, in Section \ref{configurations section}. Instead of counting subgraphs exactly, it will be more useful for us to instead count homomorphisms from a graph $H$ to the graph induced by a subset of vertices $U$ in our graph. Given a subset of vertices $U\subset V(G)$ and a fixed subgraph $H$, we will use the notation $H(U)$ to denote the number of homomorphisms from $H$ to the induced subgraph $G[U]$. For example, if $H$ is a cycle on $k$ vertices, then $C_k(U)$ denotes the number of sequences of vertices $\{v_1,\cdots, v_k\}$ such that $v_i \in U$ and $v_i\sim v_{i+1}$ for $1\leq i\leq k-1$ and $v_i\sim v_k$. Notice that this is counting labeled and possibly degenerate cycles. 
\subsection{Applications}
We now present some applications of Theorems \ref{thm1} and \ref{thm2}. There are many $(n, d, \lambda)$-graphs that can be computed explicitly in the literature, such as Cayley graphs. In this section, we only emphasize some particular ones, viz., dot-product graphs and distance graphs. 

Let $G=(V, E)$ be the dot-product graph in $\mathbb{F}_q^t$ defined by $V=\mathbb{F}_q^t\setminus (0, \ldots, 0)$, and there is an edge between two vertices $(x_1, x_2, \ldots, x_t)$ and $(y_1, y_2, \ldots, y_t)$ if $x_1y_1+\cdots+x_ty_t=1$. It is well-known in the literature that this graph is an $(n,{d}, \lambda)$-graph with $n=q^t-1$, $d\sim q^{t-1}$ and $\lambda \leq 2q^{\frac{t-1}{2}}$, see \cite[Theorem 8.1]{vinh2014sovability} for example. 

The following is the direct application of Theorems \ref{thm1} and \ref{thm2}. 

\begin{theorem}\label{app1}
    Let $U\subset \mathbb{F}_q^t$ be a subset of vertices in the dot-product graph. There is a constant $C$ such that
\begin{enumerate}
    \item If $|U|\geq C q^{\frac{t+1}{2}}$, then the dimension of  $\mathcal{H}(U)$ is at least $2$. 
    \item If $|U|\geq C \max\{q^{\frac{7t+19}{14}}, q^{t-1}\}$, then the dimension of  $\mathcal{H}(U)$ is at least $3$.
\end{enumerate}
\end{theorem}

In three dimensions, if the geometric properties are taken into consideration, then we have a better result, which improves earlier results in \cite{IMS}. 
\begin{theorem}\label{app0}
  Let $U\subset \mathbb{F}_q^3$ be a subset of vertices in the dot-product graph. There is a constant $C$ such that
\begin{enumerate}
    \item If $|U|\geq C q^{2}$, then the dimension of  $\mathcal{H}(U)$ is at least $2$. 
    \item If $|U|\geq C q^{5/2}$, then the dimension of  $\mathcal{H}(U)$ is equal to $3$.
    \end{enumerate}
\end{theorem}
This theorem recovers the exponent $3/2$ in dimension $2$ from the paper \cite{IMS} and improves the exponent $11/4$ in dimension 3 from \cite{IMS} down to $5/2$. 

We now move to the graph defined by the distance function. Let $G=(V, E)$ be the distance graph in $\mathbb{F}_q^t$ defined by $V=\mathbb{F}_q^t$, and there is an edge between two vertices $(x_1, x_2, \ldots, x_t)$ and $(y_1, y_2, \ldots, y_t)$ if $(x_1-y_1)^2+\cdots+(x_t-y_t)^2=1$. It is well-known in the literature that this graph is an $(n,{d},\lambda)$-graph with $n=q^t$, $d \sim q^{t-1}$, $\lambda\leq  2q^{\frac{t-1}{2}}$-graph, see \cite[Sections 2--6]{bannai2004finite} and \cite[Section 3]{kwok1992character} for example.
As above, we also have the following application in this setting.
\begin{theorem}\label{app2}
    Let $U\subset \mathbb{F}_q^t$ be a subset of vertices in the distance graph. There is a constant $C$ such that
\begin{enumerate}
    \item If $|U|\geq C q^{\frac{t+1}{2}}$, then the dimension of  $\mathcal{H}(U)$ is at least $2$. 
    \item If $|U|\geq C \max\{q^{\frac{7t+19}{14}}, q^{t-1}\}$, then the dimension of  $\mathcal{H}(U)$ is at least $3$.
\end{enumerate}
\end{theorem}

Compared to the result in \cite{VC2} with $t=2$, we recover their result to find VC-dimension $2$ but we can see that we get a trivial result for VC-dimension $3$ whereas in \cite{VC2} the exponent $15/8$ was given. This comes from the fact that in some specific settings, counting subgraph configurations in the next sections is unnecessary, and in the plane $\mathbb{F}_q^2$, the distance graph possesses some nice geometric properties, which implies a simpler proof with better exponents. Unlike the dot-product graph, we do not know how to improve the exponent $15/8$, which is left as an open question.



\subsection{Open questions}
We do not believe that Theorem \ref{thm2} is sharp in general. In particular, one of our estimates on subconfigurations (Theorem \ref{count-H_2} below) involves a recursive step using Cauchy-Schwarz that could probably be improved in many cases, but this was the best estimate we found. In Section \ref{configurations section} we point out other configurations for which we believe it would be interesting to tighten our estimates.


Our theorems give conditions to guarantee that the VC-dimension of an $(n,d,\lambda)$-graph is at least 2 or 3. In \cite{ABC}, for each fixed $D$ the authors give the threshold function for a random graph to have VC-dimension at least $D$. It would be interesting to do this in the setting of $(n, d,\lambda)$-graphs: that is, to give conditions on $d$ and $\lambda$ which would guarantee that the graph has VC-dimension at least $D$. New ideas would be needed to do this. Already to show a lower bound of VC-dimension $4$ using the techniques in this paper is outside of our capabilities.

Finally, in Theorem \ref{app0}, we do not have a construction showing that the exponent $5/2$ is best possible, but we believe that it could be, and it would be interesting to determine if this is correct.


\subsection{Structure}
The paper is organized as follows. In Section \ref{tools section} we collect preliminary lemmas that we will require during the proofs. In Section \ref{configurations section} we show that we can count or bound $H(U)$ for various graphs $H$ as long as $U$ is large enough. In Sections \ref{theorem 1 proof section} and \ref{theorem 2 proof section}, we show how to use these counting results to prove Theorems \ref{thm1} and \ref{thm2}. Finally, in Section \ref{application section}, we show our applications to the distance and dot-product graphs. Given two functions $f,g: \mathbb{N} \to \mathbb{N}$ we will use the notation $f\ll g$ to mean that $f = O(g)$ and $f\gg g$ to mean that $f = \Omega(g)$. 
\section{Tools}\label{tools section}
In this section we list the tools that we will need in our proofs. We will extensively use the following weighted version of the expander mixing lemma. As a historical note, the (unweighted) expander mixing lemma was proved at least as early as 1980 by Haemers in his PhD thesis (\cite{haemers} Theorem 3.1.1) in the language of design theory.

\begin{lemma}\label{th:expanderMixing}
Let $G = (V,E)$ be an $(n, d, \lambda)$-graph, and $A$ be its adjacency matrix. For real $f, g\in L^2(V)$, we have 
\[\left|\langle
f,Ag\rangle-dn\mathbb{E}(f)\mathbb{E}(g)\right|\leq \lambda\|f\|_2\|g\|_2,\]
where
\[\mathbb{E}(f):=\frac{1}{n}\sum_{v\in V}f(v), \text{ and } ||f||_2^2=\sum_{v\in V}|f(v)|^2.\]
\end{lemma}

At times, the classical expander mixing lemma will not be precise enough for our purposes. We will use the following lemma, first proved in \cite{PSTT}, that is specific to tensor powers of graphs. Notice that the function $f$ below has a different domain than $f$ above.

\begin{proposition}
\label{keylemma}
Let $G$ be an $(n, d, \lambda)$-graph. For two {non-negative} functions $f, g\colon V\times V\to \mathbb{R}$, we define $F(x)=\sum_{y}f(x, y)$, $G(z)=\sum_{w}g(z, w)$, $F'(y)=\sum_{x}f(x, y)$, and $G'(w)=\sum_{z}g(z, w)$. Then we have 
\[\left\vert \sum_{x\sim z, y\sim w}f(x, y)g(z, w)-\frac{d^2}{n^2}||f||_1||g||_1 \right\vert\le \lambda^2||f||_2||g||_2+\frac{d\lambda}{n}\left(||F||_2||G||_2+||F'||_2||G'||_2\right).\]
\end{proposition}

We will also need an extension of the expander mixing lemma to counting paths and cycles. The following can be proved using Lemma \ref{th:expanderMixing} or Proposition \ref{keylemma} and induction, see \cite{PSTT}.

\begin{proposition}[Proposition 3.5 in \cite{PSTT}]\label{paths}
Let $G$ be an $(n, d, \lambda)$-graph, $k\ge 1$ an integer,  and $U$ be a vertex set with $\lambda \cdot \frac{n}{d} = o(|U|)$. Let $P_k(U)$ denote the number of (labeled, possibly degenerate) paths with $k$ edges in $U$. Then we have 
\[P_k(U)=\left(1 + o(1) \right)\frac{|U|^{k+1}d^k}{n^k}.\]
\end{proposition}
\begin{theorem}[Theorem 1.5 in \cite{PSTT}]\label{cycle-main1}
Let $G$ be an $(n, d, \lambda)$-graph and $U$ be a vertex set with ${\lambda \cdot \frac{n}{d}=o(|U|)}$. Let $C_m(U)$ denote the number of (labeled, possibly degenerate) cycles of length $m$ with vertices in $U$. Then we have 
\[
\left|C_m(U) -\frac{|U|^md^m}{n^m}\right| = O \left( \frac{\lambda |U|^{m-1}d^{m-1}}{n^{m-1}} + \frac{\lambda^{m-2} |U|^2 d}{n}\right),
\]
\end{theorem}

The following is a straightforward technical lemma used in the arguments below, showing that subsets of pseudo-random graphs cannot have too many vertices of too large or small degree.

\begin{lemma}\label{max degree lemma}
Let $G$ be an $(n, d, \lambda)$-graph. Assume $U\subset V(G)$ with $\lambda (n/d) = o(|U|)$, then there exists a subset $U'\subset U$ such that $|U'|=(1-o(1)) |U|$ and 
\[\frac{|U|d}{2n}\le \sum_{v\in U}\sum_{u\sim v}1\le \frac{2|U|d}{n},\]
for any $u\in U'$.
\end{lemma}

\begin{proof}
Let $L\subset U$ be the subset of vertices of $U$ with more than $2\frac{|U|d}{n}$ neighbors in $U$ and $S$ the subset of vertices with fewer than $\frac{1}{2}\frac{|U|d}{n}$ neighbors in $U$. Then by Lemma \ref{th:expanderMixing} we have that 
\[
|L| \cdot 2\frac{|U|d}{n} \leq e(L,U) \leq |L||U|\frac{d}{n} + \lambda \sqrt{|L||U|},
\]
which implies that 
\[
|L| \leq \frac{\lambda^2 n^2}{d^2|U|} = o(|U|),
\]
by the assumption on the size of $U$. Similarly, $|S| = o(|U|)$ and taking $U' = U \setminus (S\cup L)$ gives the result.

\end{proof}

\section{Counting configurations}\label{configurations section}
In this section, we count several configurations which we will need in order to prove Theorems \ref{thm1} and \ref{thm2}. The most important configurations are pictured in Figure \ref{ref_label_overall}.

Let $H_1$ be the number of $5$-tuples $(x, y, z, u, v)\in U^5$ such that 
\[ x\sim  y, y\sim z, z\sim u, u\sim v, u\sim x.\]
Let $H_2$ be the number of $5$-tuples $(x, y,z, u, v)\in U^5$ such that 
\[x\sim u, u\sim z, z\sim y, x\sim y, u\sim v, v\sim y\]

That is, $H_2$ is a copy of $K_{2,3}$.  Let $H_3$ be the number of $7$-tuples $(x, y, z, u, v, u', x')\in U^7$ such that
\[x\sim y, y\sim z, z\sim u, u\sim x, u\sim v,\]\[ v\sim u', u'\sim z, u'\sim x', x'\sim y.\]

We define two more configurations which are related to $H_3$. Let $H_3^+$ be the same as configuration $H_3$ with the additional condition that $x\sim u'$. And let $H_3^-$ be the same configuration as $H_3$ with the additional restriction that $x$ and $u'$ are the same vertex. That is, $H_3^-$ is $6$ vertices $x, y, z, u, v, x'$ in $U$ satisfying
\[x\sim y, ~y\sim z, ~z\sim u, ~u\sim x, ~x\sim x', ~x'\sim y,~x\sim v, ~v\sim u, ~x\sim z.\]

\begin{figure}
    \centering
  
        \includegraphics[width=1\textwidth]{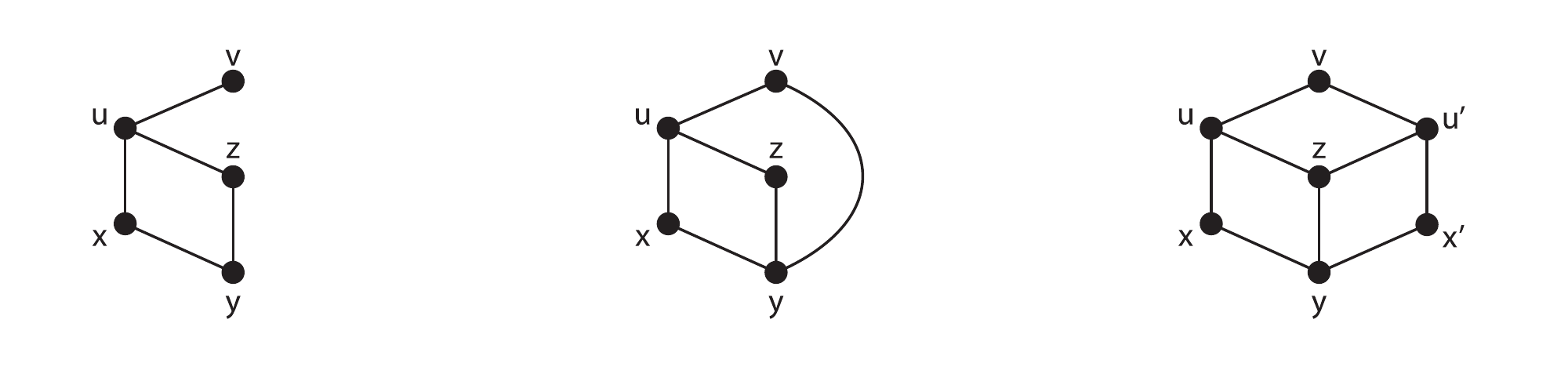}
    \caption{Configurations $H_1, H_2,$ and $H_3$}
    \label{ref_label_overall}
\end{figure}

\begin{figure}
    \centering
        \includegraphics[width=1\textwidth]{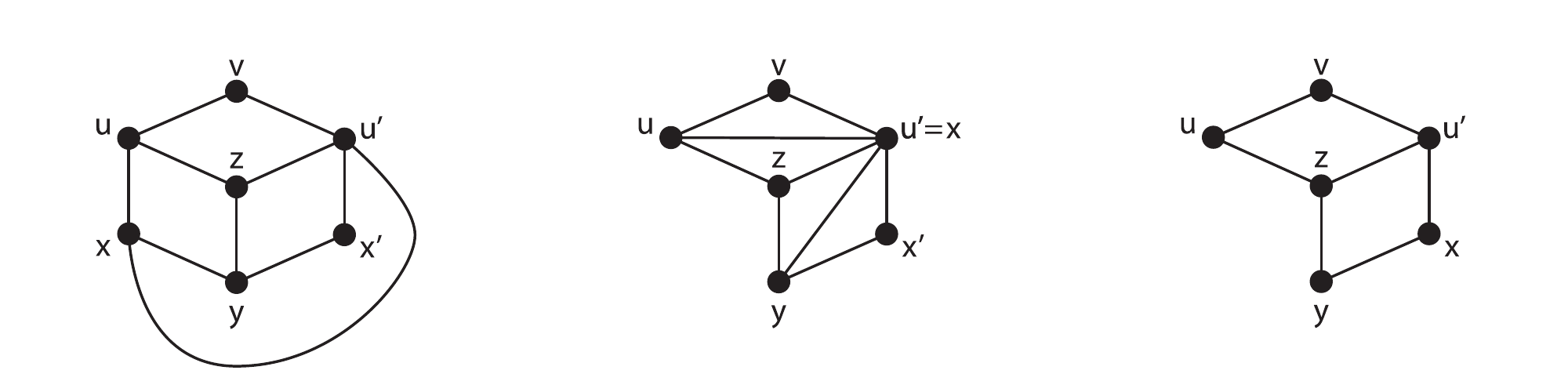}
    \caption{Configurations $H_3^+, H_3^-,$ and $H_4$}
    \label{ref_label_overall2}
\end{figure}


In this section, we show that for large enough subsets $U$, there will be a subset $U'\subset U$ with $|U'| = (1-o(1))|U|$ such that we can estimate the number of configurations in $U'$ precisely. We do not know if passing to the subset $U'$ is necessary. However, we do this so that we can bound the number of $K_{1,3}$ and $K_{1,4}$ in our subsets. We will be interested in subsets $U$ which satisfy $\lambda \frac{n}{d} = o(|U|)$. In some regimes, sets of this size do not have the expected number of stars on $3$ and $4$ edges. For example, given an $(n, d, \lambda)$ graph with $d = \sqrt{n}$, if a set $U$ has size $n^{3/4}$, then we expect $K_{1,4}(U)$ to be $\Theta\left( n^{7/4} \right)$. On the other hand, taking a set of $n^{1/4}$ vertices and their neighborhoods gives a set $U$ with at least $\Omega\left(n^{9/4}\right) $ copies of $K_{1,4}$. Note that if we wanted to apply Theorem \ref{alon counting}, we would require $\lambda \left(\frac{n}{d}\right)^4 = o(|U|)$ which is much larger than the size of sets with which we wish to work. For our application to VC-dimension, passing to a subset $U'$ does not weaken the result, and so for the trade-off of weakening the theorems in this section, we obtain in some cases the best possible results for VC-dimension.

\begin{theorem}\label{count-c1}
Let $G$ be an $(n, d, \lambda)$-graph and $U$ a subset of vertices with $ \lambda n/d = o(|U|)$. Then there is a subset $U'\subset U$ with $|U'| =(1-o(1)) |U|$ such that
\[\left\vert H_1(U')-\frac{|U'|^5d^5}{n^5} \right\vert\ll \lambda^2\frac{|U'|^3d^2}{n^2}+\lambda \frac{|U'|^4d^4}{n^4}.\]
\end{theorem}
\begin{proof}
By Lemma \ref{max degree lemma}, we may choose a subset $U'\subset U$ with $|U'| =(1-o(1)) |U|$ such that no vertex in $U'$ has degree more than $\frac{2|U|d}{n}$ in $U$. 

Define $f(x,y)$ to be the number of paths of length $2$ of the form $(x, y, u)\in (U')^3$ with $x\sim y, y\sim u$. Define $g(z,w)=1_{z\sim w},$ where the notation $1_{z \sim w}$ denotes the function that returns a 1 if $z\sim w$ and $0$ otherwise. 

Since $||f||_1$ is equal to the number of paths of length  $2$ with vertices in $U'$, it follows from Proposition \ref{paths} that
\[||f||_1= P_2(U')=(1+o(1))\frac{|U'|^3d^2}{n^2}\]
and  
\[||g||_1= P_1(U') = (1+o(1))\frac{|U'|^2d}{n},\]
under the condition that  $\lambda \frac{n}{d} = o(|U'|)$. 

We now need to compute $||f||_2, ||g||_2, ||F||_2, ||G||_2, ||F'||_2, ||G'||_2$ to apply Lemma \ref{keylemma}. We will be slightly pedantic and write out more details for the estimate of $||f||_2,$ and let the reader fill in the very similar details that we omit from the sequel.
\[||f||_2^2 = \hskip-.5em\sum_{x,y\in U'}\hskip-.5em f(x,y)f(x,y)=\hskip-.5em\sum_{x,y\in U'}|\{(x,y,u)\in (U')^3: x\sim y,y\sim u\}|^2\]
\[=\hskip-.5em\sum_{x,y\in U'}\hskip-.5em|\{(x,y,u)\in V^3: x\sim y,y\sim u\}|\cdot|\{(x,y,u')\in V: x\sim y,y\sim u'\}|.\]

So for an adjacent pair $x,y\in U',$ the contribution to $||f||_2^2$ is the number of pairs of vertices $u$ and $u'$ in $V$ such that $y$ is adjacent to both. In the case that $u$ and $u'$ are distinct, we get a 3-star. In the (degenerate) case that $u=u'$, we get a path of length 2. Note that if $x$ and $y$ are not adjacent, they will not contribute to the sum. So $||f||_2^2$ is the sum of the number of 3-stars in $G$ plus the number of 2-paths in $G.$ We now set out to estimate these quantities.

We already saw that Proposition \ref{paths}, that $P_2(U')=(1+o(1))|U'|^3d^2n^{-2}.$ So it only remains to estimate the number of 3-stars in $U'$ and verify that this is much more than $P_2(U').$ Using the fact that each vertex in $U'$ has at most $2 |U|d/n$ neighbors in $U$, we can conclude that the number of 3-stars is $O\left( |U|^4d^3n^{-3}\right)$, and this clearly dominates our estimate of $P_2(U'),$ as $\lambda n / d = o(U)$ by assumption, and $\lambda \geq 1$ (by considering the trace of $A^2$ one can see that $\lambda \gtrsim \sqrt{d}$). Putting everything together gives us the estimate
\[||f||_2^2\ll \frac{|U|^4d^3}{n^3}.\]

Moreover, by analogous arguments, neglecting similarly degenerate terms, one has
\[||g||_2^2=||g||_1^2=(1+o(1))\frac{|U'|^2d}{n}.\]
Similarly, $||F||_2^2$ is bounded by the number of paths of length $4$, and $||F'||_2^2$ is bounded by the number of stars with $4$ edges, i.e. 
\[||F||_2^2, ||F'||_2^2\ll \frac{|U|^5d^4}{n^4}.\]

We now bound $||G'||_2^2$ and $||G||_2^2$. It is clear that these are at most the number of paths of length $2$, and so we have \[||G'||_2^2, ||G||_2^2 \ll|U|^3d^2/n^2. \]

Putting these bounds to Proposition \ref{keylemma} and noting that 
\[
\sum_{x \sim z, y\sim w} f(x,y)g(z,w) = H_1(U')
\]gives us 
\[\left\vert H_1(U')-\frac{|U'|^5d^5}{n^5} \right\vert\ll \lambda^2\frac{|U'|^3d^2}{n^2}+\frac{d\lambda}{n}\frac{|U'|^4d^3}{n^3}.\]
This completes the proof.
\end{proof}

\begin{theorem}\label{count-H_2}
Let $G$ be an $(n, d, \lambda)$-graph and $U$ a subset of vertices such that $\lambda n/d = o(|U|)$. Then there is a subset $U'\subset U$ with $|U'| = (1-o(1))|U|$ such that
\[\left\vert H_2(U')-(1+o(1))\frac{|U'|^5d^6}{n^6} \right\vert\ll \lambda^2\cdot \frac{|U'|^3d^2}{n^2}.\]
\end{theorem}
\begin{proof}
By Lemma \ref{max degree lemma}, we may choose a subset $U'\subset U$ with $|U'| =(1-o(1)) |U|$ such that no vertex in $U'$ has degree more than $\frac{2|U|d}{n}$ in $U$. As before, we will show that this subset satisfies the conclusion.

We recall that $H_2(U')$ counts the number of $5$-tuples $(x, y,z, u, v)\in (U')^5$ such that 
\[x\sim u, u\sim z, z\sim y, x\sim y, u\sim v, v\sim y.\]
That is, it is the number of (labeled, possibly degenerate) $K_{2,3}$ in $U'$. For $u, y, z\in U'$, define $f_u(y)$ as the number of cycles of length $4$ with vertices in $U'$ such that $u$ and $y$ form a diagonal, define $g_u(z)=1$ if $u\sim z$ and $0$ otherwise. 
Summing over all $u$ in $U'$ and $y\sim z$, we have,
\[H_2(U')=\sum_{u\in U'}\sum_{y\sim z}f_u(y)g_u(z) = \sum_{u\in U'} \langle f_u, Ag_u\rangle.\]
Applying Lemma \ref{th:expanderMixing} for each choice of $u$ and summing gives us
\[\left\vert H_2(U')- \frac{d}{n}\sum_{u}\sum_{y\in U'}f_u(y)\sum_{z\in U'}g_u(z)\right\vert\ll \lambda\sum_{u\in U'}\left(\sum_{y}f_u(y)^2\right)^{1/2}\left(\sum_{z}g_u(z)^2\right)^{1/2}.\]
Notice that 
\[\sum_{u, y, z\in U'}f_u(y)g_u(z)=|H_1(U')|.\]
By the Cauchy-Schwarz inequality, one has 
\[\sum_{u\in U'}\left(\sum_{y}f_u(y)^2\right)^{1/2}\left(\sum_{z}g_u(z)^2\right)^{1/2}\le \left(\sum_{u, y}f_u(y)^2\right)^{1/2}\cdot \left(\sum_{u, z}g_u(z)^2\right)^{1/2}.\]

Now $\sum f_u(y)^2$ is the number of (labeled, possibly degenerate) copies of $K_{2,4}$. This can be bounded above by the number of $K_{2,3}$ times the degree of one of the vertices in the part of size $2$ (note that this is on average giving up a factor of about $\frac{d}{n}$). Using the fact that each vertex in $U'$ has at most $2 |U|d/n$ neighbors in $U$, we have
\[\sum_{u}\sum_{y}f_u(y)^2\le |H_2(U')|\cdot \frac{2|U|d}{n}.\]

Moreover, the sum $\sum_{u, z}g_u(z)^2$ equals $P_1(U')$, so it is $(1+o(1))|U'|^2d/n$ by Lemma \ref{th:expanderMixing}. 

Putting these computations together and using Theorem \ref{count-c1}, one has 
\[\left\vert H_2(U')-(1+o(1))\frac{|U'|^5d^6}{n^6} \right\vert\ll \lambda^2\cdot \frac{|U'|^3d^2}{n^2}.\]
This completes the proof.

\end{proof}

\begin{remark}\label{remark1}
The above two lemmas say that when $\lambda (n/d)^{3/2} = o(|U|)$, there is a subset $U'\subset U$ with $|U'|\gg |U|$ such that 
\[H_1(U')=(1+o(1))\frac{|U'|^5d^5}{n^5}, ~H_2(U')=o(H_1(U')).\]
\end{remark}

   In the proof of Theorem \ref{count-H_2}, if we assume that any two vertices in $U$ have at most $\gamma$ common neighbors, and $\Lambda=\min\{\gamma, 2|U|d/n\}$, then 
   \[\sum_{u, y}f_u(y)^2\le H_2(U')\cdot \Lambda.\]
This implies the following strengthened version.
\begin{theorem}\label{thm:gamma}
    Let $G$ be an $(n, d, \lambda)$-graph and $U$ a subset of vertices such that $\lambda n/d = o(|U|)$. If any two vertices in $U$ have at most $\gamma$ common neighbors, and $\Lambda=\min\{\gamma, 2|U|d/n\}$, then there is a subset $U'\subset U$ with $|U'| = (1-o(1))|U|$ such that
\[\left\vert H_2(U')-(1+o(1))\frac{|U'|^5d^6}{n^6} \right\vert\ll \lambda^2\cdot \frac{|U'|^2d}{n}\cdot \Lambda.\]

\end{theorem}

Next we give estimates on $H_3^+$, the supergraph of $H_3$ with respect to the number of copies of $H_3$, and finally we give an upper bound on $H_3^-$. We were not able to give an asymptotic formula for $H_3$ and we leave this as an interesting open problem. Later in the proof of Theorem \ref{thm2} we will give only a lower bound on the number of $H_3$ and this will be enough for our purposes. In order to estimate $H_3^+$ we will first need to estimate one more configuration.

Define $H_4(U')$ to be the number of $6$-tuples $(y, z, u, v, u', x)$ in $(U')^6$ such that 
\[y\sim z, ~z\sim u, ~u\sim v, ~v\sim u', ~u'\sim x, ~x\sim y, ~z\sim u'.\]
\begin{theorem}\label{configuration 4}
Let $G$ be an $(n, d, \lambda)$-graph and $U$ a subset of vertices such that $|U|\gg \lambda (n/d)^{3/2}$. Then there is a subset $U' \subset U$ with $|U'| = (1-o(1))|U|$ such that
\[H_4(U')\ll \frac{|U'|^6d^7}{n^7}.\]
\end{theorem}
\begin{proof}
By Theorem \ref{cycle-main1}, we have that the number of $C_4$ in $U$ is 
\[
O\left(\frac{|U|^4d^4}{n^4} \right),
\]
since $|U|\gg \lambda (n/d)^{3/2}$.

Let $L$ be the set of vertices in $U$ which are contained in at least $C|U|^3d^4/n^4$ copies of $C_4$ in $U$, and define $U' = U\setminus L$. Then for any $\epsilon > 0$, there is a $C$ such that $|U'| \geq (1-\epsilon)|U|$ and every vertex in $U'$ has at most $C|U|^3d^4/n^4$ copies of $C_4$ containing it. 

First we define 
\[f(z, u')=\#\{(u, v)\in U'\times U'\colon z\sim u, ~u\sim v, ~v\sim u', ~u'\sim z\},\]
and $g(y, x)=1$ if $y\sim x$ and $0$ otherwise. 

The strategy is to apply Proposition \ref{keylemma}. It is clear that $\sum_{z, u'\in U'}f(z, u')$ counts the number of $4$-cycles in $U'$, so we have that 
\[
\sum_{z,u'\in U'} f(z,u') = O\left( \frac{|U'|^4d^4}{n^4}\right).
\]
 Since $ \lambda (n/d) = o(|U|)$, we have $\sum_{y, z\in U'}g(y, z)=(1+o(1))|U'|^2d/n$ by Proposition \ref{paths}.

On the other hand, $\sum_{z, u'\in U'}f(z, u')^2$ is at most $H_4(U')$, and $\sum_{y, x\in U'}g(y, x)^2=\sum_{y, x\in U'}g(y, x)\ll |U'|^2d/n$. In the setting of Proposition \ref{keylemma}, we can check directly that $\sum_{z}F(z)^2$ counts the number of pairs of $4$-cycles sharing a common point, so is bounded by 

\[
O\left( \frac{|U'|^7d^8}{n^8}\right),
\]
by the definition of $U'$. Moreover, $\sum_{y}G(y)^2$ is at most the number of paths of length $2$ in $U'$, which is at most $O\left( |U'|^3d^2/n^2\right)$ by Proposition \ref{paths}. The same estimates hold for $\sum_{u'}F(u')^2$ and $\sum_{x}G(x)^2$. Putting 
these estimates together and noting that $H_4(U') = \sum_{x\sim u', y\sim z}f(x,y)g(z,u')$ gives us by Proposition \ref{keylemma} that 
\[H_4(U')\ll \frac{|U'|^6d^7}{n^7} +
\lambda^2H_4(U')^{1/2}\left(\frac{|U'|^2d}{n}\right)^{1/2}+\lambda\frac{|U'|^5d^6}{n^6}.\]
Thus, 
\[H_4(U')\ll \frac{|U'|^6d^7}{n^7}+\lambda^4\frac{|U'|^2d}{n}\ll \frac{|U'|^6d^7}{n^7}.\]
\end{proof}

With this in hand, we are able to estimate $H_3^+$

\begin{theorem}\label{count-c3}
Let $G$ be an $(n, d, \lambda)$-graph and $U$ a subset of vertices such that $|U|\gg \lambda (n/d)^{3/2}$. Then there is a subset $U' \subset U$ with $|U'| = (1-o(1))|U|$ such that
\[\left\vert H_3^+(U')- \frac{H_3(U')d}{n}\right\vert\ll \lambda \left(H_2(U')\frac{|U'|d}{n}\right)^{1/2}\left(\frac{|U'|^6d^8}{n^8}+\lambda H_2(U')\right)^{1/2}.\]
\end{theorem}
\begin{proof}
As before, by Lemma \ref{max degree lemma}, we pass to a subset $U'\subset U$ with $|U'| =(1-o(1)) |U|$ such that no vertex in $U'$ has degree more than $\frac{2|U|d}{n}$ in $U$. We furthermore may assume that the conclusion of Theorem \ref{configuration 4} holds on $U'$.

We recall that $H_3^+(U')$ counts the number of $7$-tuples $(x, y, z, u, v, u', x')\in (U')^7$ such that
\[x\sim y, y\sim z, z\sim u, u\sim x, u\sim v, v\sim u', u'\sim z, u'\sim x', x'\sim y, x\sim u'.\]
To prove this theorem, we proceed as follows. 

Given $y, z, v\in U'$ such that $y\sim z$, we define 
\[f_{y, z, v}(u'):=\#\{x'\in U'\colon x'\sim y, u'\sim z, u'\sim v, u'\sim x'\},\]
and 
\[g_{y, z, v}(x):=\#\{u\in U'\colon x\sim y, x\sim u, u\sim v, u\sim z\}.\]

Note that $f_{y,z,v}(u')$ is the number of vertices in $N(u')\cap N(y)$ times the indicator that $u' \in N(z)\cap N(v)$ and $g_{y,z,v}(x)$ is the number of vertices in $N(v)\cap N(x)\cap N(z)$ times the indicator that $x\in N(y)$.

Then we have that
\begin{align*}H_3^+(U')&=\sum_{y, z, v, y\sim z}\sum_{x\sim u'}f_{y, z, v}(u')g_{y, z, v}(x)= \sum_{y, z, v, y\sim z} \langle f, Ag \rangle.\end{align*}
We first observe that 
\[\sum_{y, z, v, y\sim z}||f_{y, z, v}||_1||g_{y, z, v}||_1=H_3(U').\]
Next, let $y\sim z$ and $v$ be fixed. For each $u'$, we have that $(f_{y,z,v}(u'))^2$ is at most the number of homomorphisms to $K_{2,3}$ where $y$ and $u'$ are in the part of size $2$ and $z$ is in the part of size $3$, times the indicator that $u'\sim v$. Therefore, this is at most the number of $K_{2,3}$ with $y,u'$ in the part of size $2$ times the degree of $u'$. This implies
\[\sum_{y, z, v, y\sim z}||f_{y, z, v}||_2^2\ll H_2(U')\cdot \frac{|U|d}{n}.\]
In the next step, we bound the sum $\sum_{y, z, v, y\sim z}||g_{y, z, v}||_2^2$ which is much more complicated. We denote this sum by $M$. We have $M$ counts the number of $6$-tuples $(y, z, u, v, u', x)\in U'$ such that 
\[y\sim z, ~z\sim u, ~u\sim v, ~v\sim u', ~u'\sim x, ~x\sim y, ~x\sim u, ~z\sim u'.\] To bound $M$ from above, we proceed as follows. More precisely, given $z\sim u'$, define 
\[k_{z, u'}(u)=\#\{v\in U'\colon v\sim u, v\sim u', z\sim u\}, \]
and 
\[h_{z, u'}(x)=\#\{y\in U'\colon y\sim x, y\sim z, x\sim u'\}.\]
Then it is clear that 
\[M=\sum_{\substack{x, u\in U' \\ x\sim u}}\sum_{\substack{z, u'\in U' \\ z\sim u'}}k_{z, u'}(u)h_{z, u'}(x).\]


To apply Lemma \ref{th:expanderMixing}, the following estimates are needed.  
\[\sum_{\substack{z, u'\in U' \\ ~z\sim u'}}|k_{z, u'}(u)|^2=H_2(U'), ~\sum_{\substack{z, u'\in U' \\ ~z\sim u'}}|h_{z, u'}(x)|^2=H_2(U').\]
We also need to bound $\sum_{z, u'}k_{z, u'}(u)h_{z, u'}(x)$. This sum counts the number of $6$-tuples $(y, z, u, v, u', x)\in U'$ such that 
\[y\sim z, ~z\sim u, ~u\sim v, ~v\sim u', ~u'\sim x, ~x\sim y, ~z\sim u'.\]
Note that this is exactly $H_4(U')$. By Theorem \ref{configuration 4}, we have that 
\[H_4(U')\ll \frac{|U'|^6d^7}{n^7},\]
since $|U'|\gg \lambda(n/d)^{3/2}$.
Applying Lemma \ref{th:expanderMixing}, we have 
\[M = \sum_{z\sim u'} \langle k_{z,u'}, Ah_{z,u'}\rangle \ll \frac{|U'|^6d^8}{n^8}+\lambda H_2(U').\]
Hence, the theorem follows from Lemma \ref{th:expanderMixing}.
\end{proof}

\begin{theorem}\label{count-c4}
Let $G$ be an $(n, d, \lambda)$-graph and $U$ a subset of vertices such that $n\lambda/d = o(|U|)$, then there is a subset $U'\subset U$ of size $(1-o(1))|U|$ such that
   \[H_3^-(U')\ll \frac{|U'|^6d^7}{n^7} + \lambda^2\frac{|U'|^4d^4}{n^4}+\lambda \frac{|U'|^5d^5}{n^5}+\lambda^2 \frac{|U'|^4d^{7/2}}{n^{7/2}}+\lambda^3\frac{|U'|^3d^{5/2}}{n^{5/2}}+\lambda^4\frac{|U'|^2d}{n}.\]
\end{theorem}

\begin{proof}
As before, pass to a subset $U'$ of size $(1-o(1))|U'|$ where each vertex in $U'$ has at most $\frac{2|U|d}{n}$ neighbors in $U$. 

Fix $z\in U'$. Define $f_z(x)$ to be the number of tuples $(z, x, v, u)\in U'^4$ such that $z\sim x, x\sim v, v\sim u, u\sim z$, and $g_z(x')$ to be the number of tuples $(z, x', y)\in U'^3$ such that $z\sim y, y\sim x'$. 

Then it is clear that 
\[H_3^-(U')\le \sum_{z\in U}\sum_{x\sim x'}f_z(x)g_z(x').\]
We first note that the sum $\sum_{z}||f_z||_1||g_z||_1$ is equal to the number of tuples $(z, x, v, u, y, x')$ such that 
\[z\sim x,~ x\sim v,~ v\sim u,~ u\sim z,~ z\sim y,~ y\sim x'.\]
This number is at most $H_1(U')\cdot 2|U|d/n$. To bound $H_3^-(U')$ from above, we will apply Lemma \ref{th:expanderMixing}. To proceed further, we need to bound the $L^2$-norm of $f_z$ and $g_z$. In the way we define the function $f_z$ and $g_z$, it is not hard to see that the sum $\sum_{z}||f_z||_2^2$ is at most the number of $6$-cycles $C_6(U')$, which by Theorem \ref{cycle-main1} is bounded from above by 
\[(1+o(1))\frac{|U'|^6d^6}{n^6}+O\left(\frac{\lambda^4|U'|^2d}{n}\right).\]
Similarly, we also have the sum $\sum_{z}||g_z||_2^2$ is at most the number of $4$-cycles $C_4(U')$, which is bounded from above by 
\[(1+o(1))\frac{|U'|^4d^4}{n^4}+O\left(\frac{\lambda^2|U'|^2d}{n}\right).\]
In other words, we can conclude that 
\[H_3^-(U')\le \frac{H_1(U')|U'|d^2}{n^2}+\lambda \frac{|U'|^5d^5}{n^5}+\lambda^2 \frac{|U'|^4d^{7/2}}{n^{7/2}}+\lambda^3\frac{|U'|^3d^{5/2}}{n^{5/2}}+\lambda^4\frac{|U'|^2d}{n}.\]
Applying Theorem \ref{count-c1} completes the proof.
\end{proof}

\begin{remark}
In practice, configurations in Theorems \ref{count-c3} and \ref{count-c4} might not exist. 
\end{remark}


\section{Proof of Theorem \ref{thm1}}\label{theorem 1 proof section}
To prove Theorem \ref{thm1}, we need to show that there exists a set of two distinct points, say $x_1, x_2\in U$, such that the restriction of $\mathcal{H}$ to $U$ yields every possible function from $\{x_1, x_2\}$ to $\{0, 1\}$. This is equivalent to saying that there exists $u^*, u_{1}, u_2, u_{12}\in U$ such that 
\[x_1\sim u_{12},x_2\sim u_{12},\]
\[x_1\sim u_1, x_2\not\sim u_1,\]
\[x_1\not\sim u_2,x_2\sim u_2,\]
and 
\[x_1\not\sim u^*,x_2\not\sim u^*.\]

We first prove the existence of $u_{1}, u_2, u_{12}$ which are joined to $x_1$ and $x_2$ in the prescribed way. By Lemma \ref{max degree lemma} and Theorem \ref{count-c1}, we may choose a subset $U'\subset U$ with $|U'| =(1-o(1)) |U|$ such that no vertex in $U'$ has degree more than $\frac{2|U|d}{n}$ in $U$ and that satisfies 
\[
\left\vert H_1(U')-\frac{|U'|^5d^5}{n^5} \right\vert\ll \lambda^2\frac{|U'|^3d^2}{n^2}+\lambda\frac{|U'|^4d^4}{n^4}.
\]We will show that $x_1, x_2, u_1, u_2, u_{12}$ may be found in $U'$. Afterwards, we will show that by the definition of $U'$ there will be a vertex $u^*\in U$ which is adjacent to neither $x_1$ nor $x_2$. Hence it suffices to show the existence of these $5$ vertices in $U'$ with the prescribed edges and non-edges.

We do this by finding a path on $5$ vertices $u_1\sim x_1\sim u_{12}\sim x_2\sim u_2$. Such a path will satisfy our requirements if and only if we have $u_1\not\sim x_2$ and $u_2\not\sim x_1$. In order to guarantee the existence of such a path we count all paths and then count paths where at least one of the forbidden edges is present. 

We know from Proposition \ref{paths} that the number of paths on $5$ vertices in $U'$ with it at least 
\begin{equation}\label{number of 4 paths}
    (1-o(1))\frac{|U'|^5 d^4}{n^4}, 
\end{equation}

because of the restriction that $\frac{\lambda n}{d} = o(|U|)$. Note that there are also this many paths where all $5$ vertices are distinct, again by Proposition \ref{paths}. Such a path will satisfy our requirements unless $x_1\sim u_2$ or $x_2\sim u_1$. If either of these forbidden edges occurs, we have found configuration $H_1$.

By the definition of $U'$ we have that 
\[
H_1(U') \leq \frac{|U'|^5d^5}{n^5} + O\left( \lambda^2\frac{|U'|^3d^2}{n^2}+\frac{\lambda|U'|^4d^4}{n^4}\right),
\]
This quantity is much smaller than the number of paths in \eqref{number of 4 paths} by the hypothesis on the size of $U$. Choosing any vertex $u^*\in U$ which is not in $N(x_1)\cup N(x_2)$ completes the proof. We are guaranteed that such a vertex exists, as $x_1,x_2\in U',$ so neither has more than $\frac{2|U|d}{n}$ neighbors in $U'$. So $|N(x_1)\cup N(x_2)|\leq \frac{4|U|d}{n},$ leaving at least $|U'|-\frac{4|U|d}{n}$ choices for $u^*$ in $U'\setminus (N(x_1)\cup N(x_2)).$ The number of choices of $u^*$ is then strictly positive by our assumptions.

\section{Proof of Theorem \ref{thm2}}\label{theorem 2 proof section}
The proof of Theorem \ref{thm2} is a bit involved, so we first give a sketch.
\subsection{Sketch of proof}

To prove the VC-dimension is at least $3$, we need to find vertices $x_1, x_2, x_3, y_1, y_2, y_3, y_{12}, y_{13}, y_{23}, y_{13}, y^*\in U$ with $x_1, x_2, x_3$ distinct such that 
\begin{enumerate}
    \item $x_1\sim y_{123}$, $x_2\sim y_{123}$, $x_3\sim  y_{123}$
    \item $x_1\sim y_{12}$, $x_{2}\sim y_{12}$, $x_3\not\sim y_{12}$, and similarly for $y_{13}$ and $y_{23}$. 
    \item $x_1\sim y_1$, $x_2\not\sim y_1$, $x_3\not\sim y_1$, and similarly for $y_2$ and $y_3$. 
    \item $x_1\not\sim y^*$, $x_2\not\sim y^*$, $x_3\not\sim y^*$. 
\end{enumerate}

If we can find vertices in such a configuration, then the set $\{h_{x_1}, h_{x_2}, h_{x_3}\}$ is shattered by $\mathcal{H}(U)$ and this shows that the VC-dimension of $U$ is at least $3$. We again note that for the definition of shattering, it does not matter what adjacency or non-adjacency we see between $x_i$ and $x_j$ or between any of the $y$ vertices. We also note that for the definition of shattering to be satisfied, it is possible that some $y$ and some $x_j$ are the same vertex. However, we will avoid this situation because the possible presence of loops in our graphs make it hard to analyze. To prove that we can find a set of vertices with the above properties, we use the count of configurations in Section \ref{configurations section}. 

To prove Theorem \ref{thm2}, we do the following steps. 

{\bf Step 1:} We first pass to a subset $U'$ where we may count the configurations from Section \ref{configurations section}. Then we argue that it is enough to find vertices $x_1, x_2, x_3, y_{123}, y_{12}, y_{23}, y_{13}$ in $U'$.

{\bf Step 2:} Next we count the number of configurations $H_1$ and use Cauchy-Schwarz to lower bound the number of tuples satisfying

\begin{itemize}
	\item $x_1\sim  y_{123}, x_2 \sim y_{123}, x_3 \sim y_{123}$
	\item $x_1 \sim y_{12}, x_2 \sim y_{12}$
	\item $x_1 \sim y_{13}, x_3\sim y_{13}$
	\item $x_2 \sim y_{23}, x_3 \sim y_{23}$
	\item $x_2 \not\sim y_{13}$
	\item $x_1 \neq x_2$, $x_3 \neq x_2$
	\item $y_{123} \neq y_{12}$, $y_{123} \neq y_{23}$
\end{itemize}

{\bf Step 3:}

In the previous step, some of the vertices may overlap. Since we must have that $x_1\not=x_3$, next we show that the number of configurations in Step $2$ with $x_1=x_3$ is much smaller than the total.

{\bf Step 4:}

After the step $3$, we get a lower bound on the number of tuples satisfying
\begin{itemize}
	\item $x_1\sim  y_{123}, x_2 \sim y_{123}, x_3 \sim y_{123}$
	\item $x_1 \sim y_{12}, x_2 \sim y_{12}$
	\item $x_1 \sim y_{13}, x_3\sim y_{13}$
	\item $x_2 \sim y_{23}, x_3 \sim y_{23}$
	\item $x_2 \not\sim y_{13}$
	\item $x_1 \neq x_2$, $x_3 \neq x_2$, $x_1\ne x_3$
	\item $y_{123} \neq y_{12}$, $y_{123} \neq y_{23}$
\end{itemize}

{\bf Step 5:}
For those tuples in the Step $4$, we need to remove tuples with $y_{12}\sim x_3$ and $y_{23}\sim x_1$ and tuples with $x_2 = y_{12}$ or $x_3 = y_{12}$. If there is an adjacency we have the configuration $H_3^{xu '}$ and if there is an overlap of vertices we have the configuration $H_3^-$.

{\bf Step 6:}

Using the assumption in the statement of the theorem, we conclude the proof.
\subsection{Details}

\begin{proof}[Proof of Theorem \ref{thm2}]
First we pass to a subset $U'$ where our configurations are well-behaved. By Lemma \ref{max degree lemma}, we may consider a subset $U'\subset U$ such that all vertices in $U'$ have between $\frac{1}{2}\frac{|U|d}{n}$ and $2\frac{|U|d}{n}$ neighbors in $U$. 
{By Remark \ref{remark1} and Theorem \ref{count-c3}, we may assume that $H_1(U')$ is close to its expected value, that $H_2(U') = o(H_1(U'))$, and that $H_3^+(U') \leq \epsilon(|U|^7d^9/n^9 + H_3(U'))$ where $\epsilon$ can be taken arbitrarily small by choosing $C$ large enough in the hypotheses of the theorem. } 

We focus on the existence of the vertices $x_1, x_2, x_3, y_{123}, y_{12}, y_{23}, y_{13}$ in $U'$. If we can find such a subconfiguration, then by the assumption in the hypotheses of the theorem, we may find vertices $y_1, y_2, y_3$. Since $|U|$ is much larger than $\frac{|U|d}{n}$ and the vertices in $U'$ have at most $2\frac{|U|d}{n}$ neighbors in $U$, we may also find the vertex $y^*$ which completes the configuration, as in the final step of the proof of Theorem \ref{thm1}. Therefore it suffices to find the vertices $x_1, x_2, x_3, y_{123}, y_{12}, y_{23}, y_{13}$ in $U'$ satisfying the adjacency and non-adjacency conditions. Note that this is exactly configuration $H_3$.

Since each vertex in $U'$ has at most $2|U|d/n$ neighbors, the number of $3$-stars is at most $8|U|^4d^3/n^3$. Under the condition $\lambda n/d = o(|U|)$, we also know from Proposition \ref{paths} that the number of $3$-paths is $(1+o(1))|U|^4d^3/n^3$ and from Theorem \ref{cycle-main1} that $C_4(U) = (1+o(1))\frac{|U|^4 d^4}{n^4}$. For any $\{x,y,z,u,v\}$ with $x\sim y, y\sim z, z\sim u, u\sim x, u\sim v$ (that is, that form configuration $H_1$), if $x=z$ we have a $P_3$, if $y=u$ we have a star on $3$ edges, if $v\in \{x,y,z\}$ we have a $4$-cycle, and if $y\sim v$ we have configuration $H_2$. 

So, it follows from Remark \ref{remark1} that  the number of tuples in $U'$ satisfying
\begin{equation}\label{c1 prime definition}x\sim y, y\sim z, z\sim u, u\sim x, u\sim v, y\not\sim v, x\ne z, y\ne u, v\not\in \{x,y,z\}\end{equation}
is at least $(1-\epsilon)|H_1|$ as long as 
$|U'|\geq C\lambda (n/d)^{3/2}$, for any $\epsilon$ and a large enough constant $C$.

We denote the set of those tuples by $H_1'$. We will define the notation $H_1'(x, y, z, u, v)$ as the indicator that $x,y,z,u,v$ satisfy \eqref{c1 prime definition}, that is that they form a configuration in $H_1'$.
Define 
\[f(y, z, v)=\sum_{x, u\in U'}H_1'(x, y, z, u, v).\]
Then we have that 
\[
\sum_{y,z,v, ~y\sim z}(f(y,z,v))^2 = H_3(U').
\]
We will complete the proof by bounding this number from below and then showing that the number of these homomorphisms which are either not injective or which have an edge that is forbidden is smaller than the total number. Proceeding with this, we have
\[\frac{|U|^{10}d^{10}}{n^{10}}\ll|H_1'|^2=\left(\sum_{y, z, v\in U, ~y\sim z}f(y, z, v)\right)^2\le \sum_{y,z,v} f(y, z, v)^2\cdot \sum_{y, z, v}1_{y\sim z}.\]
Notice that 
\[\sum_{y, z, v}1_{y\sim z}\ll |U'|\cdot \frac{|U'|d}{n}\cdot |U'|, \]
by the degree restriction on $U'$. This means that 
\[\sum_{y,z,v} f(y, z, v)^2\gg \frac{|U'|^7d^9}{n^9}.\]

Given a fixed $y\sim z$ and $v$, the quantity $(f(y,z,v))^2$ counts the number of $4$-tuples $(x,u,x',u')$ so that $H_1'(x,y,z,u,v) = H_1'(x',y,z,u',v) = 1$. In Figure \ref{ref_label_overall}, we think of $y,z,v$ as playing the roles of $x_1, y_{123}, y_{23}$ respectively, and we hope for a $4$-tuple $(x,u,x',u')$ that play the roles of $(y_{13}, x_3, y_{12}, x_2)$. By the definition of $H_1'$, we have that each $\{x,y,z,u,v\}$ forming a copy of $H_1'$ satisfies $x\not=z$, $y\not=u$, $v\not\in \{x,y,z\}$, and $y\not\sim v$. 

This means that the number of tuples $(x_1, x_2, x_3, y_{12}, y_{13}, y_{23}, y_{123})\in (U')^7$ such that 
\begin{itemize}
    \item $x_1\sim y_{123}, x_2\sim y_{123}, x_3\sim y_{123}$
    \item $x_1\sim y_{12}, x_2\sim y_{12}$
    \item $x_1\sim y_{13}, x_3\sim y_{13}$
    \item $x_2\sim y_{23}, x_3\sim y_{23}$
    \item $x_1\ne x_2$, $x_1\ne x_3$
    \item $y_{123}\ne y_{12}$, $y_{123}\ne y_{13}$, $y_{123}\ne y_{23}$.
    \item $y_{23} \not\in \{y_{13}, y_{12}, x_1, x_2\}$
    \item $x_1\not\sim y_{23}$
\end{itemize}
is at least $\Omega\left( |U|^7d^9/n^9\right)$. Such a $7$-tuple satisfies all of our conditions unless one of the following occurs. 

\begin{itemize}
    \item $y_{12}=y_{13}$
    \item $x_2 = x_3$
    \item $x_2 = y_{13}$
    \item $x_3 = y_{12}$
    \item $x_2 \sim y_{13}$
    \item $x_3 \sim y_{12}$
    
\end{itemize}

If $x_2 \sim y_{13}$ or $x_3\sim y_{12}$, we get a configuration isomorphic to $H_3^+$. Hence by Theorem \ref{count-c3}, the number of these is at most $\epsilon(|H_3(U')|)$ whenever 
\begin{equation}\label{two}|U'|\geq C\lambda (n/d)^{13/7},\end{equation}

for any $\epsilon$ and a large enough $C$. If $x_2 = x_3$, then the vertices $\{x_1, x_2, y_{12}, y_{13}, y_{123}\}$ form configuration $H_2$ and furthermore $y_{23}\sim x_2$. Thus the number of $6$-tuples of this form is bounded above by $|H_2(U')|\cdot \frac{2|U|d}{n}$. By Theorem \ref{count-H_2}, this is at most 

\[
O\left( \frac{|U|^6d^7}{n^7} + \lambda^2 \frac{|U|^4 d^3}{n^3}\right).
\]
This is smaller than $\epsilon |U'|^7d^9/n^9$ when 
\begin{equation}\label{three}|U|\geq C\lambda^{2/3}(n/d)^2,\end{equation}
for a large enough absolute constant $C$. Similarly, if $y_{12} = y_{13}$, then the vertices $\{y_{12}, y_{123}, x_1, x_2, x_3\}$ form a $H_2$ and additional $y_{23}$ is adjacent to $x_2$ (and $x_3)$. Hence the number of $6$-tuples of this form is also bounded above by 
\[
O\left( \frac{|U|^6d^7}{n^7} + \lambda^2 \frac{|U|^4 d^3}{n^3}\right).
\]

Finally, if $x_2 = y_{13}$ or equivalently if $x_3 = y_{12}$, then we have configuration $H_3^-$. By Theorem \ref{count-c4}, the number of these is at most 
\[
O\left( \frac{|U'|^6d^7}{n^7} + \lambda^2\frac{|U'|^4d^4}{n^4}+\lambda \frac{|U'|^5d^5}{n^5}+\lambda^2 \frac{|U'|^4d^{7/2}}{n^{7/2}}+\lambda^3\frac{|U'|^3d^{5/2}}{n^{5/2}}+\lambda^4\frac{|U'|^2d}{n}\right),
\]
which is smaller than $|U'|^7d^9/n^9$ when 
\begin{equation}\label{four}
|U| \geq C\max\{\lambda^{2/3}(n/d)^2, \lambda (n/d)^{3/2}, ~\lambda^{3/4}(n/d)^{13/8}, ~\lambda^{4/5}(n/d)^{8/5} \}\geq C\max\left\lbrace \lambda^{2/3}(n/d)^2, \lambda (n/d)^{13/7}\right\rbrace.
\end{equation}

Therefore if $C$ is chosen to be a large enough absolute constant, we have that there is a $7$-tuple such that none of the forbidden adjacencies or vertex identifications occur, and therefore the VC-dimension of $\mathcal{H}(U)$ is at least $3$ under 
\[|U| \geq C\max\left\lbrace \lambda^{2/3}(n/d)^2, \lambda (n/d)^{13/7}\right\rbrace.\]

\end{proof}




\section{Proofs of Theorems \ref{app1}, \ref{app0}, and \ref{app2}}\label{application section}
The proofs of Theorems \ref{app1} and \ref{app2} are almost the same, so we only present one of them.
\begin{proof}[Proof of Theorem \ref{app1}]
The only detail we need to check is the first condition, namely, given any three vertices $v_1, v_2, v_3$ in $U$, we can find three vertices $u_1, u_2, u_3$ in $U$ such that $u_i\cdot v_j=1$ if and only if $i=j$. 

Using the fact that any two distinct hyperplanes in $\mathbb{F}_q^t$ intersect in at most $q^{t-2}$ points. This means that $|N(v_i)\cap N(v_j)|\le q^{t-2}$. On the other hand, by Lemma \ref{max degree lemma} we may pass to a subset $U'$ such that for each $i$, $|N(v_i) \cap U'| = \Theta( q^{-1}|U|)$. So, the condition $|U|\geq Cq^{t-1}$ is enough to guarantee that we will find such vertices $v_1, v_2, v_3$ as long as $C$ is chosen large enough.     
\end{proof}

\begin{proof}[Proof of Theorem \ref{app0}]
In this particular setting, we can use the geometric properties to improve the argument in the proof of Theorem \ref{thm2}. We will run the same proof as in Theorem \ref{thm2} except we also remove all unit vectors from $U$ when defining $U'$. Since there are at most $O(q^2)$ unit vectors in $\mathbb{F}_q^3$, we still have that $|U'| = (1+o(1))|U|$. Now we use geometric properties to show that for some of the configurations, we either have better upper bounds or they do not exist.

More precisely, 
 when estimating copies of configuration 2, the fact that the intersection of two distinct planes is either a line or a null set, we have that any two vertices in the graph have at most $q$ common neighbors. We may therefore use Theorem \ref{thm:gamma} and this implies that the condition (\ref{three}) is replaced by
\[|U'|\geq Cq^{5/2}.\]
Next we show that any copy of $H_3^-$ contains a unit vector and hence by the definition of $U'$ we may ignore the configuration. Assume there is a copy of $H_3^-$ with labels as in Figure \ref{ref_label_overall}, so $x=u'$. By the adjacencies, we have that the planes $\{w: v\cdot w=1\}$, $\{w: x\cdot w=1\}$, and $\{w: z\cdot w=1\}$ all contain $u$. Since the intersection of three planes is either a line or is empty, the intersection must be a line. But the line containing $u$ and $x$ is the intersection of the planes $\{w: v\cdot w=1\}$ and $\{w: z\cdot w=1\}$, and hence the line containing $x$ and $u$ is also the intersection of all three planes. But this now implies that $x\cdot x = 1$. Hence, the condition (\ref{four}) is not needed. 

With a similar argument and the fact that $x_1\not\sim y_{23}$, we can check that the configuration $H_3^+$ does not appear (this argument appears at the end of the paper \cite{IMS}). So the condition (\ref{two}) is not needed. 

In conclusion, in this setting, we only need the condition $|U'|\geq C q^{5/2}$ to ensure that the dimension is at least three. 

To see the upper bound on the VC-dimension in $\mathbb F_q^3$, notice that no quadruple of points can be shattered, as any three of them determine a plane. Since dot-products are constant on planes, any quadruple of points that have the same non-zero dot product with a point will lie on a plane. Thus any triple of them will lie on the same plane. So if any three of the points from the quadruple determine a given non-zero dot product with a point in $\mathbb F_q^3,$ so will the fourth, and we cannot shatter any set of four points.

This completes the proof. 
\end{proof}

\section{Acknowledgements}
T. Pham would like to
thank to the VIASM for the hospitality and for the excellent working conditions. M. Tait was partially supported by National Science Foundation grant DMS-2011553.

\bibliographystyle{amsplain}

\bibliography{bib}
\end{document}